\documentclass[12pt,reqno]{amsart}

\usepackage{amsfonts,amsbsy,amssymb,amstext,mathrsfs,latexsym,color,url}

\newtheorem{theorem}{Theorem}[section]
\newtheorem{lemma}[theorem]{Lemma}

\newtheorem{corollary}[theorem]{Corollary}

\newtheorem{definition}[theorem]{Definition}

\newtheorem{remark}[theorem]{Remark}
\newtheorem{example}[theorem]{Example}

\newcommand{\bitem}{\begin{itemize}}
\newcommand{\eitem}{\end{itemize}}
\newcommand{\benum}{\begin{enumerate}}
\newcommand{\eenum}{\end{enumerate}}
\newcommand{\beq}{\begin{equation}}
\newcommand{\eeq}{\end{equation}}

\newcommand{\spann}{\mbox{\rm span}}

\newcommand{\supp}{{\text{\rm supp}} \,}

\def\NN{\mathbb{N}}

\def\RR{\mathbb{R}}

\def\CC{\mathbb{C}}

\begin{document}

\title{Spectral Tetris Fusion Frame Constructions}

\author[P.~G.~Casazza]{Peter~G.~Casazza}
\address[P.~G.~Casazza]{Department of Mathematics, University of Missouri, Columbia, Missouri 65211, USA; E-mail: casazzap@missouri.edu}
\author[M.~Fickus]{Matthew Fickus}
\address[M.~Fickus]{Department of Mathematics and Statistics, Air Force Institute of Technology, Wright-Patterson Air Force Base, Ohio 45433, USA; E-mail: Matthew.Fickus@afit.edu}
\author[A.~Heinecke]{Andreas Heinecke}
\address[A.~Heinecke]{Department of Mathematics, University of Missouri, Columbia, Missouri 65211, USA; E-mail: ah343@mail.mizzou.edu}
\author[Y.~Wang]{Yang Wang}
\address[Y.~Wang]{Department of Mathematics, Michigan State University, East Lansing, Michigan 48824, USA; E-mail: ywang@math.msu.edu}
\author[Z.~Zhou]{Zhengfang Zhou}
\address[Z.~Zhou]{Department of Mathematics, Michigan State University, East Lansing, Michigan 48824, USA; E-mail: zfzhou@math.msu.edu}

\begin{abstract}
Spectral tetris is a flexible and elementary method to construct unit norm frames with a given frame operator, having all of its eigenvalues greater than or equal to two. One important application of spectral tetris is the construction of fusion frames. We first show how the assumption on the spectrum of the frame operator can be dropped and extend the spectral tetris algorithm to construct unit norm frames with any given spectrum of the frame operator.  We then provide a sufficient condition for using this generalization of spectral tetris to construct fusion frames with prescribed spectrum for the fusion frame operator and with prescribed dimensions for the subspaces. This condition is shown to be necessary in the tight case of redundancy greater than two.
\end{abstract}

\keywords{Frames, fusion frames, spectral tetris}
\thanks{PGC, MF and AH were supported by DTRA/NSF 1042701 and AFOSR F1ATA00183G003.  Additional support was provided by NSF DMS 1008183 (PGC, AH) and AFOSR F1ATA00083G004, F1ATA0035J001 (MF), NSF DMS-0813750, DMS-08135022 and DMS-1043034 (YW), NSF DMS-0968360 (ZZ).  The views expressed in this article are those of the authors and do not reflect the official policy or position of the United States Air Force, Department of Defense, or the U.S. Government}

\maketitle

\section{Introduction}

A fusion frame is a sequence of subspaces of a Hilbert space along with a sequence of weights, so that the sequence of weighted orthogonal projections onto these subspaces sums to an invertible operator on the space. Fusion frames, introduced in \cite{CK04} and refined in \cite{CKL08}, have become a subject of interest due to their applicability to problems in distributed processing, sensor networks and a host of other directions. Fusion frames provide resilience to noise and erasures due to, for instance, sensor failures or buffer overflows \cite{Bod07,CK07b,KPCL08,PKC08}, as well as robustness to subspace perturbations \cite{CKL08}, which can happen because of imprecise knowledge of sensor network topology.  For fusion frame applications, we generally need extra structure on the fusion frame, such as prescribing the fusion frame operator or the dimensions of the subspaces, or both.

In this paper we address the question of how to efficiently construct fusion frames with prescribed dimensions of the subspaces and prescribed eigenvalues of the fusion frame operator. Our main tool is the spectral tetris construction for unit norm frames, which we will review in section \ref{prelim}. This construction is limited to the case of frames having a frame operator with spectrum in $[2,\infty)$ and therefore we will first,
in section \ref{GST}, extend spectral tetris to the most general possible case in terms of the prescribed spectrum of the frame operator of the unit norm frame. Precisely, we give a version of spectral tetris, capable of constructing an $M$-element unit norm frame in $\CC^N$ with prescribed eigenvalues $(\lambda_n)_{n=1}^N\subseteq(0,\infty)$ satisfying only the necessary trace condition $\sum_{n=1}^N\lambda_n=M$. Before tackling this case in section \ref{GST2}, we first consider the case of tight frames in section  \ref{GST1}, i.e. we extend the existing construction to unit norm tight frames of redundancy less than two.

In section \ref{greater2}, we give a necessary condition on a prescribed sequence of eigenvalues and dimensions, under which we can use the generalized spectral tetris method to construct a fusion frame having those eigenvalues for its fusion frame operator and having those dimensions for its subspaces. We further show that this condition is also necessary in the case of unit norm tight frames of redundancy of at least two.
\section{Preliminaries and notation}\label{prelim}

\subsection{Fusion frames}

Let $(e_n)_{n=1}^N$ be the standard unit vector basis of $\RR^N$.
The {\it synthesis operator} of a finite sequence
$(f_m)_{m=1}^M\subseteq\CC^N$ is $F\colon\CC^M\to\CC^N$ given by
\[
Fg=\sum_{m=1}^M\langle g,e_m\rangle f_m,
\]
i.e. $F$ is the $N\times M$ matrix whose $m$-th column is $f_m$. The sequence $(f_m)_{m=1}^M$ is a {\it frame} if its {\it frame operator} $S=FF^*$ satisfies $AI\leq S\leq BI$ for some positive constants $A,B$, where $I$ is the identity on $\CC^N$. In particular, the spectrum of $S$ is positive real.
The sequence is a {\it tight frame} if $A=B$, i.e. if 
\[
Af=\sum_{m=1}^M\langle f,f_m\rangle f_m
\]
for all $f\in\mathbb{C}^N$, or equivalently if
\[
\sum_{m=1}^M\langle f_m,e_n\rangle \overline{\langle f_m,e_{n'}\rangle}=
\begin{cases}
A, & n=n',\\
0, & n\neq n'.
\end{cases}
\]
In the case of a tight frame, the constant $A$ equals $M/N$ and is also called the {\it tight frame bound} or the {\it redundancy} of the frame. The synthesis matrix of the frame $(f_m)_{m=1}^M$ is called \emph{$s$-sparse}, if it has 
$s$ nonzero entries.
A {\it unit norm tight frame} is a tight frame $(f_m)_{m=1}^M$ for which $\|f_m\|=1$ for all $m=1,\ldots,M$. Unit norm tight frames provide Parseval-like decompositions in terms of nonorthogonal vectors of unit norm.
If $(f_m)_{m=1}^M$  is a unit norm frame, the operators $f\mapsto\langle f,f_m\rangle f_m$ arising in the frame operator
\[
Sf=\sum_{m=1}^M\langle f,f_m\rangle f_m
\]
are rank-one orthogonal projections. Fusion frame theory is the study of sums of projections with weights and of arbitrary rank. In particular, a sequence $(W_k,v_k)_{k=1}^K$ of subspaces $(W_k)_{k=1}^K$ of $\CC^N$ and weights $(v_k)_{k=1}^K$ is a {\it fusion frame} if the sequence $(P_k)_{k=1}^K$ of orthogonal projections onto those subspaces satisfies
\[
AI\leq\sum_{k=1}^Kv_k^2P_k\leq BI
\]
for some positive constants $A,B$.  It is a {\it tight fusion frame} if $A=B$. Using the Horn-Klyachko compatibility  inequalities, \cite{MRS10} gives a characterization of the sequences of weights and dimensions of the subspaces for which tight fusion frames exist. This result however can hardly be used in practice, since it involves the computation of the Littlewood-Richardson coefficients of certain associated partitions and moreover does not give a construction of the existing fusion frames.
 In this paper we restrict ourselves to the case where all weights are equal to one and denote the fusion frame by 
$(W_k)_{k=1}^K$.
In this case, the {\it fusion frame operator} is
 $S=\sum_{k=1}^KP_k$. If $D_k$ is the dimension of the range of $P_k$
and $(f_{k,d})_{d=1}^{D_k}$ is an orthonormal basis of the range of $P_k$ then
\[
Sf=\sum_{k=1}^KP_kf=\sum_{k=1}^K\sum_{d=1}^{D_k}\langle f,f_{k,d}\rangle f_{k,d}
\]
for all $f\in\CC^N$. This shows that every fusion frame arises from a classical frame that satisfies additional orthogonality requirements. To be precise, we say that a sequence \smash{$(f_{k,d})_{k=1,d=1}^{K\quad D_k}\subseteq\CC^N$} \textit{generates} a fusion frame $(W_k)_{k=1}^K$ with $\dim W_k=D_k$ for $k=1\ldots,K$ if \smash{$(f_{k,d})_{k=1,d=1}^{K\quad D_k}$} is a frame for $\CC^N$ and $(f_{k,d})_{d=1}^{D_k}$ is orthonormal for $k=1\ldots,K$.

We say that a (fusion) frame has a certain spectrum or certain eigenvalues, if its (fusion) frame operator has this spectrum, respectively these eigenvalues. Note that we are always counting multiplicities of the eigenvalues.
For any frame, the sum of its eigenvalues equals the sum of the lengths of its vectors, i.e. the number of vectors in case we
are dealing with unit norm frames.

Building on the theory of majorization and the Schur-Horn Theorem,~\cite{CFMPS11} 
shows how to explicitly construct every possible frame whose frame operator has a given arbitrary spectrum and whose vectors are of given arbitrary lengths. In this paper we use a much easier algorithm for the construction of unit norm frames
with given spectrum. This algorithm is called spectral tetris and we will discuss it in detail below. We focus our attention on this algorithm, since it is tailor made for the construction of fusion frames, as it constructs frames of vectors, a lot of which are orthogonal to each other due to disjointness of their supports.

\subsection{Spectral Tetris}

The term \textit{spectral tetris} refers to the first systematic method for constructing unit norm tight frames. This construction was introduced in \cite{CFMWZ09} to generate unit norm tight frames in $\RR^N$ for any dimension $N$ and any number of frame vectors $M$, provided that $M\geq 2N$. Choosing all weights to equal $1$,~\cite{CFMWZ09} provides a complete characterization of triples $(N,K,d)$ for which tight fusion frames of $K$ subspaces of equal dimension $d$ exist in $\RR^N$ and gives an elegant algorithm to produce such tight fusion frames for most of those triples.

A straightforward extension to the construction of unit norm frames having a desired frame operator with eigenvalues $(\lambda_n)_{n=1}^N\subseteq[2,\infty)$ satisfying $\sum_{n=1}^N\lambda_n=M$ was introduced in \cite{CCHKP10}. For convenience we review this construction in table~\ref{fig:originalSTC} and will refer to it as the {\em spectral tetris construction} (STC). STC constructs synthesis matrices with unit norm columns, whose rows are pairwise orthogonal and
square sum to the desired eigenvalues. This way, the frame operator is a diagonal matrix, having precisely the desired eigenvalues on its diagonal.
We strongly recommend a glance at \cite{CFMWZ09} or  \cite{CCHKP10} for instructive examples on how the algorithm constructs the desired synthesis matrices by using $2\times 2$ building blocks of the form
\begin{align}\label{2by2}
\begin{bmatrix}
\sqrt{x}&\sqrt{x}\\
\sqrt{1-x}&-\sqrt{1-x}
\end{bmatrix}.
\end{align}
A construction for equi-dimensional fusion frames having eigenvalues as above for their fusion frame operators is given in \cite{CCHKP10}. The sufficient condition for this construction to work, is that the dimension $d$ of the subspaces satisfies $\sum_{n=1}^N\lambda_n=dK$ where $K$ is the number of subspaces and $\lambda_n\leq K-3$ for $n=1,\ldots,N$.

\begin{table}[h]
\centering
\framebox{
\begin{minipage}[h]{6.0in}
\vspace*{0.3cm}
{\sc \underline{STC: Spectral Tetris Construction}}

\vspace*{0.4cm}

{\bf Parameters:}\\[-3ex]
\begin{itemize}
\item Dimension $N\in\NN$.
\item Number of frame elements $M\in\NN$.
\item Eigenvalues $(\lambda_n)_{n=1}^N \subseteq[2,\infty)$ such that
$\sum_{n=1}^N\lambda_n=M$.
\end{itemize}

{\bf Algorithm:}\\[-3ex]
\begin{itemize}
\item[1)] Set $k = 1$.
\item[2)] For $j=1,\ldots,N$ do
\item[3)] \hspace*{1cm} Repeat
\item[4)] \hspace*{2cm} If $\lambda_j < 1$ then
\item[5)] \hspace*{3cm} $f_k = \sqrt{\frac{\lambda_j}{2}} \cdot e_j + \sqrt{1-\frac{\lambda_j}{2}} \cdot e_{j+1}$.
\item[6)] \hspace*{3cm} $f_{k+1} = \sqrt{\frac{\lambda_j}{2}}\cdot e_j - \sqrt{1-\frac{\lambda_j}{2}} \cdot e_{j+1}$.
\item[7)] \hspace*{3cm} $k = k+2$.
\item[8)] \hspace*{3cm} $\lambda_{j+1} = \lambda_{j+1} - (2-\lambda_j)$.
\item[9)] \hspace*{3cm} $\lambda_j = 0$.
\item[10)] \hspace*{2cm} else
\item[11)] \hspace*{3cm} $f_k = e_j$.
\item[12)] \hspace*{3cm} $k = k+1$.
\item[13)] \hspace*{3cm} $\lambda_j = \lambda_j - 1$.
\item[14)] \hspace*{2cm} end.
\item[15)] \hspace*{1cm} until $\lambda_j = 0$.
\item[16)] end.
\end{itemize}

{\bf Output:}\\[-3ex]
\begin{itemize}
\item Unit norm frame $(f_k)_{k=1}^M\subseteq\RR^N$.
\end{itemize}
\vspace*{0.1cm}
\end{minipage}
}
\vspace*{0.2cm}
\caption{The STC algorithm for constructing a unit norm frame with prescribed spectrum in 
$[2,\infty)$.}
\label{fig:originalSTC}
\end{table}

In this paper, we extend the spectral tetris construction to the case of arbitrary eigenvalues for the frame operator and use those frames for the construction of fusion frames with given fusion frame operators and subspaces of not necessarily equal dimensions. 

\begin{definition}\label{spectraltetrisframe}
A frame constructed via the spectral tetris construction STC is called a
\emph{spectral tetris frame}. A fusion frame $(W_k)_{k=1}^K$ is called a \emph{spectral tetris fusion frame} if there is a partition of a spectral tetris frame $(f_{k,d})_{k=1,d=1}^{K\quad D_k}$ such that $(f_{k,d})_{d=1}^{D_k}$ is an orthonormal basis for $W_k$ for every $k=1,\ldots,K$.
\end{definition}

Aside from the fact that spectral tetris frames are easy to construct, their major advantage for applications is the sparsity of their synthesis matrices. This sparsity is dependent on the ordering of the given sequence of eigenvalues for which STC is performed. Note that the original form of the algorithm in \cite{CCHKP10} assumes the sequence of eigenvalues to be in decreasing order. However, this assumption was made only for classification reasons, and it is easily seen that it can be dropped. The sparsest synthesis matrices are achieved if the sequence of eigenvalues $(\lambda_n)_{n=1}^N$ is ordered {\it blockwise}, i.e. if for any permutation $\pi$ of $\{1,\ldots,N\}$ the set of partial sums $\{\sum_{j=1}^s\lambda_j\colon s=1,\ldots,N\}$ contains at least as many integers as the set $\{\sum_{j=1}^s\lambda_{\pi(j)}\colon s=1,\ldots,N\}$. It has been shown in \cite{CHKK10}, that spectral tetris frames are optimally sparse in the sense that given $M\geq 2N$ and a sequence of eigenvalues $(\lambda_n)_{n=1}^N\subseteq[2,\infty)$, the synthesis matrix of the spectral tetris frame having these parameters is sparsest in the class of all synthesis matrices of unit norm frames that have these parameters, provided STC is run for the sequence $(\lambda_n)_{n=1}^N$ rearranged to be ordered blockwise. Note that for tight frames all eigenvalues are equal, and as such, questions of rearranging the order of the eigenvalues do not arise.

\section{Spectral tetris for arbitrary prescribed spectra}\label{GST}

The spectral tetris construction STC discussed in section \ref{prelim} is capable of constructing unit norm frames, all of
whose eigenvalues are at least $2$. It constructs the synthesis matrices of such frames by successively filling an appropriate sized matrix with ones and $2\times 2$ blocks of the form (\ref{2by2}). If we drop the condition on the spectrum to be contained in $[2,\infty)$, we can in general no longer use just $2\times 2$ blocks for a spectral-tetris-like construction. We will instead work with larger building blocks, constructed from appropriate sized discrete Fourier transform matrices.

\begin{definition}
Given $L\in\NN$, let $\omega=\exp(\frac{2\pi i}{L})$ be a primitive $L$-th root of unity. The (non-normalized)
\emph{discrete Fourier transform (DFT) matrix} in $\mathbb{C}^{L\times L}$ is defined by
\[
F_L=\left(\omega^{jk}\right)_{j,k=0}^{L-1}.
\]
\end{definition}

Let us point out that here we do not normalize $F_L$ by the factor $1/\sqrt{L}$, thus every entry of $F_L$ is of modulus one.
The importance of DFT matrices to our construction stems from the fact that they have orthogonal rows and columns, and that all entries have the same modulus. In the course of the construction we will have to alter the row norms of the DFT matrices by multiplying rows with appropriate constants. While this will destroy the pairwise orthogonality of the columns, it will preserve the pairwise orthogonality of the rows, which is the crucial feature for our construction to work. 

In all the constructions that follow, whenever a $2\times 2$ DFT matrix is used, one might as well use a $2\times 2$ matrix
of the form (\ref{2by2}). It is however not obvious how one could work with real matrices of size larger than $2\times 2$ instead of complex DFT matrices.

It is perhaps most instructive to first look at an example of the construction we are going to introduce in this section.
\begin{example}\label{LRSTCexample}
We construct a $5$-element unit norm tight frame in $\CC^4$. In the following we will use the notation
$\omega_L=\exp(\frac{2\pi i}{L})$.

We can start filling the desired $4\times 5$ synthesis matrix with an altered $2\times 2$ DFT matrix in the upper left corner. The alteration we make, is to multiply the entries of the first row by \smash{$\sqrt{5/8}$} in order to make the first row have the desired norm $\sqrt{5/4}$. To get normalized columns, we multiply the second row of the $2\times 2$ DFT matrix by 
$\sqrt{3/8}$. At this point we have constructed the first row and the first two columns of the desired synthesis matrix:
\begin{align*}
\begin{bmatrix}
\sqrt{\frac{5}{8}}&\sqrt{\frac{5}{8}}&0&0&0\\
\sqrt{\frac{3}{8}}&\sqrt{\frac{3}{8}}\cdot \omega_2&\cdot&\cdot&\cdot\\
0&0&\cdot&\cdot&\cdot\\
0&0&\cdot&\cdot&\cdot\\
\end{bmatrix}.
\end{align*} 
Note that so far we have constructed a matrix whose first two rows are orthonormal, no matter how we keep filling the second row. The second row at this point has norm $\sqrt{3/4}$, while we need to make it have norm $\sqrt{5/4}$. We can not insert another alterated $2\times 2$ DFT matrix in the same fashion as above, since we would have to multiply its first row by the factor
$\sqrt{2/8}$ in order to have the second row of the synthesis matrix square sum to $5/4$, and its second row by the factor
$\sqrt{3/4}$ to get normalized columns. But then the second row of the altered DFT block would already square sum to 
$3/2$ and thus exceed what we desire for the row norms of our synthesis matrix.
We can however, alter a $3\times 3$ DFT matrix in the above fashion. We multiply its first column by $\sqrt{1/6}$ to add the missing weight to the second row of the desired synthesis matrix and make it square sum to $5/4$. Moreover, we multiply the second and the third row of the $3\times 3$ DFT matrix by $\sqrt{5/12}$ to get
\begin{align*}
\begin{bmatrix}
\sqrt{\frac{5}{8}}&\sqrt{\frac{5}{8}}&0&0&0\\
\sqrt{\frac{3}{8}}&\sqrt{\frac{3}{8}}\cdot \omega_2&\sqrt{\frac{1}{6}}&\sqrt{\frac{1}{6}}&\sqrt{\frac{1}{6}}\\
0&0&\sqrt{\frac{5}{12}}&\sqrt{\frac{5}{12}}\cdot \omega_3&\sqrt{\frac{5}{12}}\cdot \omega_3^2\\
0&0&\sqrt{\frac{5}{12}}&\sqrt{\frac{5}{12}}\cdot \omega_3^2&\sqrt{\frac{5}{12}}\cdot \omega_3^4\\
\end{bmatrix}.
\end{align*} 
The latter matrix is the synthesis matrix of the desired frame, since its columns are normalized and its rows are pairwise orthogonal and square sum to $5/4$.
\end{example}

We would like to extend definition~\ref{spectraltetrisframe} to include frames like the one in example~\ref{LRSTCexample}, i.e. in addition we call all frames that are constructed using altered DFT matrices as building blocks of their synthesis matrices \emph{spectral tetris frames}, and fusion frames constructed by using the vectors of any such spectral tetris frames
\emph{spectral tetris fusion frames}.

\subsection{Tight low redundancy spectral tetris frames}\label{GST1}

We now give a version of spectral tetris for the construction of unit norm tight frames of redundancy smaller than $2$, i.e. our goal is to construct an $M$-element unit norm tight frame for $\CC^N$ where $N<M<2N$. 
In most cases, the original construction STC does no longer work if $M<2N$. Indeed, it is shown in~\cite{CHKWZ11}, that in case
$N<M<2N$ and $M$ and $N$ relatively prime, STC can construct a unit norm tight frame of $M$ vectors in $N$ dimensions, if and only if $M=2N-1$.
Since, it is thus no longer possible to always work with $2\times2$ building blocks and ones, we will first address the question: What size alterated DFT-matrices should we use as building blocks of the to be constructed synthesis operator, in order to achieve the best possible sparsity?
We want to develop an algorithm following the lines of example \ref{LRSTCexample}, i.e. we want to compose the desired synthesis matrix out of square blocks whose first and last rows successively overlap.
If we use $K$ square matrices as building blocks of the $N\times M$-synthesis matrix, then due to the overlapping, we have
\[
M=N+(K-1).
\] 
We will first check the best possible sparsity we can get from this set-up, without considering at this moment the question, whether we can alter the DFT blocks in such a way, that the row norms of the synthesis matrix equal $\sqrt{M/N}$ and that the columns are normalized.
Note that,
if we do a spectral-tetris-like construction as in example \ref{LRSTCexample} with $K$ altered DFT matrices of sizes $m_1,\ldots,m_K$,
then the sparsity of the synthesis matrix we get is
$
\sum_{k=1}^Km_k^2.
$

\begin{lemma}\label{Pete}
Let $m_1\leq\cdots\leq m_k,$ and $M$ be positive integers.
The minimizer of 
$\sum_{k=1}^Km_k^2$, subject to the constraints
$\sum_{k=1}^Km_k=M$, and $1\le m_k \le K$ for $k=1,\dots,K$,
is given by
\[ m_1= \cdots =m_r=L,\ \ \mbox{and}\ \ m_{r+1}=\cdots =m_K=L+1,\]
where $L = \lfloor \frac{M}{K}\rfloor$ and $r=K(L+1)-M$.
\end{lemma}

\begin{proof}
We first show that the minimizer above satisfies 
\begin{align}\label{step1}
m_{K}\le m_1+1.
\end{align}
Indeed, assume we have achieved the
minimum sum above and $m_K-m_1 \ge 2$.   Then
\begin{eqnarray*}
(m_1+1)^2 + (m_K-1)^2 &=&  m_1^2+m_K^2 -2(m_k-m_1-1)\\
&<& m_1^2+m_K^2.
\end{eqnarray*}
Therefore, replacing $m_1$ and $m_K$ by $m_1+1$ and $m_K-1$, respectively, we would get a set of numbers satisfying the constraints of the lemma but having smaller square sum. Thus we have shown (\ref{step1}).

Successive application of (\ref{step1}) implies, that there is an $L$ and
an index $1\le r \le K$, so that
\[ m_1= \cdots = m_r = L, \mbox{ and }\\ m_{r+1}= \cdots
= m_K = L+1.\]
Hence,
\[
M =\sum_{r=1}^K m_r= r L+(K-r)(L+1)=K(L+1) - r.
\]
Thus, $r=K(L+1)-M$ and
since $1\le r \le K$, it follows that $L = \lfloor \frac{M}{K} \rfloor$.
\end{proof} 

While STC used building blocks of sizes $1\times 1$ (namely ones) and $2\times 2$, lemma \ref{Pete} suggests that we work with DFT matrices of the sizes $L\times L$ and $(L+1)\times(L+1)$ as building blocks to construct sparse unit norm tight frames of redundancy smaller than two.
Here are the alterations we have to perform on those DFT matrices.

\begin{definition}
Let $N<M<2N$ and $L=\lfloor\frac{M}{M-N+1}\rfloor$. We denote a matrix by $D_L$, or call it a \emph{$D_L$ block}, if it is derived from the DFT matrix $F_L$ by multiplying the entries of the $j$-th row of $F_L$ by $\sqrt{\frac{c_j}{NL}}$, where $c_2=\cdots=c_{L-1}=M$ and 
$1\leq c_1,c_L\leq M$ are integers, and if it has normalized columns. We call $c_1$ and $c_L$ the \emph{first}, respectively the \emph{last} \emph{correction factor} of $D_L$. In the same way, replacing $L$ by $L+1$, we define $D_{L+1}$, i.e. we multiply the $j$-th row of $F_{L+1}$ by $\sqrt{\frac{c_j}{N(L+1)}}$.
\end{definition}

Note that the row norms of a $D_L$ or $D_{L+1}$ block equal $\sqrt{M/N}$, except possibly for their first and last row. We need to collect some information about the first and last correction factors.

\begin{lemma}\label{correctionrange}
Let $N<M<2N$ and $L=\lfloor\frac{M}{M-N+1}\rfloor$. 
\begin{itemize}
\item[(i)]
An integer $c_1$ is a first correction factor of a $D_L$ block, if and only if $L(N-M)+M\leq c_1\leq M$.
\item[(ii)] If $L(N-M)+N>0$, then any $L(N-M)+N\leq c_1\leq M$ is a first correction factor of a $D_{L+1}$ block. 
If $L(N-M)+N\leq 0$, then any $L(N-M)+N\leq c_1\leq L(N-M)+M$ is a first correction factor of a $D_{L+1}$ block. 
\end{itemize}
\end{lemma}
\begin{proof}
(i) Let $c_1$ be the first and $c_L$ the last correction factor of a $D_L$ block. We need to show $1\leq c_L \leq M$. Since $D_L$ has normalized columns, we have
\[
c_L=NL-(L-2)M-c_1\geq NL-(L-1)M=L(N-M)+M>0,
\]
where the last inequality holds since $L<\frac{M}{M-N}$.
On the other hand, $c_L\leq M$ implies
\[
NL-(L-2)M-c_1\leq M,
\]
and thus $L(N-M)+M\leq c_1$.

(ii) Let $c_1$ be the first and $c_{L+1}$ the last correction factor of the $D_{L+1}$ block.
The condition that the columns are normalized, is equivalent to
\[
c_{L+1}=N(L+1)-(L-1)M-c_1.
\]
Thus $c_{L+1}\leq M$ is equivalent to 
\[
N(L+1)-(L-1)M-c_1\leq M,
\]
i.e.
\[
L(N-M)+N\leq c_1.
\]
If $L(N-M)+N>0$, then 
\[
c_{L+1}=N(L+1)-(L-1)M-c_1\geq N(L+1)-(L-1)M-M=L(N-M)+N>0.
\]
Now assume $L(N-M)+N\leq 0$ and $c_1<L(N-M)+M$. Then 
\[
c_{L+1}=N(L+1)-(L-1)M-c_1\geq N(L+1)-(L-1)M-L(N-M)-M=N>0.
\]
\end{proof}


\begin{definition}
Let $c_1,c_{L_1}$ be the first, respectively last correction factor of a $D_{L_1}$ block and 
$d_1,d_{L_2}$ be the first, respectively last correction factor of a $D_{L_2}$ block. If 
$c_{L_1}+d_1=M$ then we call $c_1-d_1$ the \emph{stepsize of the $D_{L_1}$ block}.
\end{definition}

\begin{lemma}\label{stepsizes}
Let $N<M<2N$ and $L=\lfloor\frac{M}{M-N+1}\rfloor$. 
The step size of a $D_L$ block is  $L(N-M)+M$,
while the stepsize of a $D_{L+1}$ block is $L(N-M)+N$.
\end{lemma}

\begin{proof}
Consider first the case of a $D_L$ block with first  correction factor $c_1$ and last correction factor
$c_L$.
By the normalization of the columns we have $c_L=LN-c_1-(L-2)M$.
So the step size is
\[
c_1- (M-c_L)=L(N-M)+M.
\]
In the case of a  $D_{L+1}$ block with correction factors $c_1$ and $c_{L+1}$, we have by the column normalization
$c_{L+1}=N(L+1)-c_1-(L-1)M$, and thus for the step size
\[
c_1- (M-c_L)=L(N-M)+N.
\]
\end{proof}

In case $\lfloor\frac{M}{M-N+1}\rfloor=1$ we want to use the above mentioned result from~\cite{CHKWZ11}, which characterizes the cases in which STC works, given that $N<M<2N$. To this end, we make the following observation.

\begin{lemma}\label{klemma}
Let $N<M<2N$ be positive integers. Then $ M=2N-1$, if and only if
$
\left\lfloor\frac{M}{M-N+1}\right\rfloor=1.
$
\end{lemma}
\begin{proof}
Let $1\leq k\leq N-1$ such that $M=2N-k$. If $k=1$, then
\[
\left\lfloor\frac{M}{M-N+1}\right\rfloor=\left\lfloor\frac{2N-1}{N}\right\rfloor=1.
\]
If $k\geq 2$, then
\[
\left\lfloor\frac{M}{M-N+1}\right\rfloor=\left\lfloor\frac{2N-k}{N-(k-1)}\right\rfloor\geq2.
\]
\end{proof}

\begin{table}[h]
\centering
\framebox{
\begin{minipage}[h]{6.in}
\vspace*{0.3cm}
{\sc \underline{TDFTST: Tight DFT Spectral Tetris}}

\vspace*{0.4cm}

{\bf Parameters:}
\begin{itemize}
\item Dimension $N\in\NN$.
\item Number of frame elements $M\in\NN$, where $N<M<2N$.
\item $L=\left\lfloor\frac{M}{K}\right\rfloor$, where $K=M-N+1$.
\end{itemize}

\vspace*{0.2cm}

{\bf Algorithm:}
\begin{itemize}
\item[1)] If $L=1$ 
\item[2)]\hspace*{0.5cm}	run STC with $\lambda_n=\frac{M}{N}$ for $n=1,\ldots,N$.
\item[3)] else 
\item[4)]\hspace*{0.5cm}	If $L(N-M)+N\leq0$
\item[5)]\hspace*{1cm}	run subroutine 1.
\item[6)]\hspace*{0.5cm}	else
\item[7)]\hspace*{1cm}	run subroutine 2.
\item[8)]\hspace*{0.5cm} end.
\item[9)]end.
\end{itemize}

\vspace*{0.2cm}

{\bf Output:}
\begin{itemize}
\item Unit norm tight frame $(f_m)_{m=1}^M\subseteq\CC^{N}$.
\end{itemize}
\vspace*{0.1cm}
\end{minipage}
}
\vspace*{0.2cm}
\caption{The TDFTST algorithm for constructing a unit norm tight frame of redundancy less than two.}
\label{fig:TDFTST}
\end{table}

\begin{table}[h]
\centering
\framebox{
\begin{minipage}[h]{6.in}
\vspace*{0.3cm}

{\bf Subroutine 1:}
\begin{itemize}
\item[1)] $m=0, n=1, x=M$
\item[2)] Repeat
\item[3)]\hspace*{0.5cm}	If $x\geq L(N-M)+M$ then
\item[4)]\hspace*{1cm}	$\omega=\exp(\frac{2\pi i}{L})$
\item[5)]\hspace*{1cm}	For $r=m+1,\ldots,m+L$ do
\item[6)]\hspace*{1.5cm}	$
					f_r=\sqrt{\frac{x}{NL}}e_{n}
					+\sqrt{\frac{M}{NL}}\sum_{j=n+1}^{n+L-2}\omega^{(j-n)(r-m-1)}e_j
				$
\item[]\hspace*{1.5cm}	
				$\phantom{f_r=}
					+\sqrt{\frac{NL-M(L-2)-x}{NL}}\omega^{(L-1)(r-m-1)}e_{n+L-1}
				$
\item[7)]\hspace*{1cm}	end.
\item[8)]\hspace*{1cm}	$m=m+L$
\item[9)]\hspace*{1cm} $n=n+L-1$
\item[10)]\hspace*{1cm}	$x=x-L(N-M)-M$
\item[11)]\hspace*{0.5cm}   else
\item[12)]\hspace*{1cm}	$\omega=\exp(\frac{2\pi i}{L+1})$
\item[13)]\hspace*{1cm}	For $r=m+1,\ldots,m+L+1$ do
\item[14)]\hspace*{1.5cm}$	 
					f_r=	\sqrt{\frac{x}{N(L+1)}}e_{n}
					+\sqrt{\frac{M}{N(L+1)}}\sum_{j=n+1}^{n+L-1}\omega^{(j-n)(r-m-1)}e_j
				$
\item[]\hspace*{1.5cm}
				$\phantom{f_r=}
					+\sqrt{\frac{N(L+1)-M(L-1)-x}{N(L+1)}}\omega^{L(r-m-1)}e_{n+L}
				$
\item[15)]\hspace*{1cm}	end.
\item[16)]\hspace*{1cm} $m=m+L+1$
\item[17)]\hspace*{1cm} $n=n+L$
\item[18)]\hspace*{1cm} $x=x-L(N-M)-N$
\item[19)]\hspace*{0.5cm}   end.
\item[20)] until $x=0$.
\end{itemize}
\vspace*{0.1cm}
\end{minipage}
}
\vspace*{0.2cm}
\caption{Subroutine 1 used in the TDFTST algorithm.}
\end{table}

\begin{table}[h]
\centering
\framebox{
\begin{minipage}[h]{6.in}
\vspace*{0.3cm}

{\bf Subroutine 2:}
\begin{itemize}
\item[1)] $m=0, n=1, z=0$
\item[2)]Repeat
\item[3)]\hspace*{0.5cm}	$\omega=\exp(\frac{2\pi i}{L})$
\item[4)]\hspace*{0.5cm}	For $r=m+1,\ldots,m+L$ do
\item[5)]\hspace*{1cm}	$
					f_r=\sqrt{\frac{x}{NL}}e_{n}
					+\sqrt{\frac{M}{NL}}\sum_{j=n+1}^{n+L-2}\omega^{(j-n)(r-m-1)}e_j
				$
\item[]\hspace*{1cm}	
				$\phantom{f_r=}
					+\sqrt{\frac{NL-M(L-2)-x}{NL}}\omega^{(L-1)(r-m-1)}e_{n+L-1}
				$
\item[6)]\hspace*{0.5cm}	end.
\item[7)]\hspace*{0.5cm}	$m=m+L$
\item[8)]\hspace*{0.5cm} $n=n+L-1$
\item[9)]\hspace*{0.5cm}	$x=x-L(N-M)-M$
\item[10)]\hspace*{0.5cm}	$z=z+1$
\item[11)]until $z=K(L+1)-M$
\item[12)]Repeat
\item[13)]\hspace*{0.5cm}	$\omega=\exp(\frac{2\pi i}{L+1})$
\item[14)]\hspace*{0.5cm}	For $r=m+1,\ldots,m+L+1$ do
\item[15)]\hspace*{1cm}$	 
					f_r=	\sqrt{\frac{x}{N(L+1)}}e_{n}
					+\sqrt{\frac{M}{N(L+1)}}\sum_{j=n+1}^{n+L-1}\omega^{(j-n)(r-m-1)}e_j
				$
\item[]\hspace*{1cm}
				$\phantom{f_r=}
					+\sqrt{\frac{N(L+1)-M(L-1)-x}{N(L+1)}}\omega^{L(r-m-1)}e_{n+L}
				$
\item[16)]\hspace*{0.5cm}	end.
\item[17)]\hspace*{0.5cm} $m=m+L+1$
\item[18)]\hspace*{0.5cm} $n=n+L$
\item[19)]\hspace*{0.5cm} $x=x-L(N-M)-N$
\item[20)]\hspace*{0.5cm} $z=z+1$
\item[21)]until $z=K$.
\end{itemize}

\vspace*{0.1cm}
\end{minipage}
}
\vspace*{0.2cm}
\caption{Subroutine 1 used in the TDFTST algorithm.}
\end{table}

We are now ready for the main result of this section. Given $N<M<2N$, we propose the algorithm TDFTST
(table~\ref{fig:TDFTST}) to construct an $M$-element unit norm tight frame in $\CC^N$. We analyze this algorithm in the 
proof of the following theorem.

\begin{theorem} Let $M,N$ be relatively prime and $N<M<2N$. Let $K=M-N+1$, 
$L=\lfloor\frac{M}{K}\rfloor$ 
and $ r=K(L+1)-M$.
Then the synthesis matrix of the sparsest $M$-element spectral tetris unit norm tight frame in $\mathbb{C}^N$ consists of 
$r$ blocks $D_L$  and $(K-r)$  blocks $D_{L+1}$ and can be constructed by DFTSTC. 
Its sparsity is $rL^2+(K-r)(L+1)^2$.
\end{theorem}

\begin{proof}
If $L=1$, the algorithm uses STC to construct the sparsest unit norm tight frame. This is possible since by lemma \ref{klemma}, $L=1$ is equivalent to $M=2N-1$, which by~\cite{CHKWZ11} is the characterization of the cases STC can be used, given that $M<2N$.

Now suppose $L\geq2$. In this case the algorithm builds the synthesis matrix of the desired frame as in example
\ref{LRSTCexample} by inserting $D_L$ and $D_{L+1}$ blocks. The last row of a block and the first row of the following block appear in the same row of the synthesis matrix and, while the last correction factor is chosen to ensure normalized columns, the following first correction factor is chosen to guarantee that the rows of the synthesis matrix square sum to 
$M/N$. 
As the algorithm progresses, we keep track of the development of the first correction factors in the variable $x$. We have to ensure that, whenever the algorithm is determined to insert a certain block, the variable $x$ that we computed in order to
make the rows of the synthesis matrix square sum to $M/N$, lies in the range of integers that are possible as first correction factor for a block of the desired size, see lemma \ref{correctionrange}. 
The algorithm starts by letting the first correction factor $x$ of the first block to be inserted into the synthesis matrix equal $M$. The difference between a first correction factor used in the course of the algorithm and the subsequent first correction factor that is being used, is 
the step size determined in lemma \ref{stepsizes}. We have to ensure that eventually
the sum of the step sizes equals $M$, i.e. 
the last correction factor of the final block inserted into the synthesis matrix equals $M$ (in other words, the next first correction factor would be zero, but we have arrived at the point were the algorithm terminates).

Let us begin by recording that the step sizes of $r$ blocks $D_L$ and $(K-r)$  blocks $D_{L+1}$ add up to $M$. Indeed,
\begin{align}\label{gleichung}
r(L(N-M)+M)+(K-r)(L(N&-M)+N)\nonumber\\
&=(KL-r)(N-M)+KN\nonumber\\
&=(KL-K(L+1)+M)(N-M)+KN\\
&=M(N-M+K)\nonumber\\
&=M.\nonumber
\end{align}

To show the above mentioned properties, we now separately consider the cases of $L(N-M)+N$ being positive
and negative.

We begin with the case $L(N-M)+N\leq 0$: In this case the algorithm first inserts a $D_{L}$ block with first correction factor $x=M$. It then tracks the development of the first correction factors by subtracting the step size $L(N-M)+M$, respectively
$L(N-M)+N$, whenever a $D_L$, respectively $D_{L+1}$ block, has been inserted. The algorithm inserts a $D_L$ block whenever possible, that is, whenever the first correction factor at a certain run of the repeat loop is at least $L(N-M)+M$. If and only if during a run of the repeat loop the first correction factor is less then $L(N-M)+M$, the algorithm inserts a $D_{L+1}$ block and subtracts the step size $L(N-M)+N$ from the first correction factor in order to get the first correction factor for the next run of the repeat loop. Note that, if the correction factor $x$ falls below $L(N-M)+M$, it will still be positive and thus can serve as first correction factor for a $D_{L+1}$ block, since by lemma \ref{correctionrange} it suffices that 
$L(N-M)+N\leq x<L(N-M)+M$ and since $L(N-M)+N\leq0$.
Since $L(N-M)+N\leq0$, according to (\ref{gleichung}), the algorithm will thus terminate after inserting $r$ blocks $D_L$ and $(K-r)$  blocks $D_{L+1}$ in some order.

Now consider the case $L(N-M)+N>0$: In this case the algorithm fills the synthesis matrix by first using $r$ blocks $D_L$ followed by $(K-r)$ blocks $D_{L+1}$, in this order.
By (\ref{gleichung}) we have 
\[
M-r(L(N-M)+M)=(K-r)(L(N-M)+N)\geq 0,
\]
and thus
\begin{align*}\label{pos}
M-(r-1)(L(N-M)+M)\geq L(N-M)+M,
\end{align*}
which implies that starting with correction factor $x=M$ and inserting $(r-1)$ blocks $D_L$, i.e.  
subtracting $(r-1)$ step sizes $L(N-M)+M$ from $x$, results in an integer greater than or equal to $L(N-M)+M$, which therefore can serve as the first correction factor of the last $D_L$ block that is being inserted by the algorithm. After inserting this last $D_{L}$ block, $(K-r)$ blocks $D_{L+1}$ are being inserted. By (\ref{gleichung}), the final first correction factor prior to the termination of the algorithm is $L(N-M)+N$, i.e the final last correction factor is $M$. Thus the algorithm succeeds in constructing the desired synthesis matrix.
\end{proof}

\begin{remark}
If we drop the assumption on the number $M$ of vectors and the dimension $N$ of the space to be relatively prime, 
we can easily construct a unit norm tight frame by using $\gcd(M,N)$ copies of a matrix that TDFTST produces. 
To be precise, given positive integers $N<M<2N$, let $M'=M/\gcd(M,N)$ and $N'=N/\gcd(M,N)$.
Let $F$ be the $N'\times M'$ synthesis matrix that TDFTST generates for $M'$ vectors in $N'$ dimensions. Then the $N\times M$ matrix
\begin{align*}
\begin{bmatrix}
F&  &          &  \\
  &F&          &  \\  
  &  &\ddots&  \\ 
  &  &          &F
\end{bmatrix},
\end{align*} 
built from $\gcd(M,N)$ copies of $F$ as its diagonal and filled with zeros elsewhere, is the synthesis matrix of an
$M$-element unit norm tight frame in $\CC^N$. Note that this is the sparsest unit norm tight frame one can construct by using altered DFT matrices as building blocks.

It is this construction that we refer to as $TDFTST$ in case that the number of vectors and the dimension are not relatively prime.
\end{remark}


\subsection{General spectral tetris frames}\label{GST2}

We are now dropping any conditions on the sequence of prescribed eigenvalues, besides the frame condition,
implying that the spectrum is positive, and the trace condition, implying that the sum of the eigenvalues has to equal the number of frame vectors of the unit norm frame to be constructed. We discuss a version of spectral tetris, which is capable of constructing a unit norm frame of $M$ elements in $\CC^{N}$ with such spectrum, for any $M\geq N$. Note that now
questions of how to order the eigenvalues come into play again and that therefore we will in general not get the sparsest
frames possible from a spectral tetris like construction.
The version of spectral tetris for this most general set up is the algorithm DFTST
presented in table~\ref{fig:mostgeneral}.
DFTST is a combination of the ideas developed so far and we will analyze it in the proof of theorem
\ref{finaltheorem}.

\begin{table}[h]
\centering
\framebox{
\begin{minipage}[h]{6.in}
\vspace*{0.3cm}
{\sc \underline{DFTST: DFT Spectral Tetris}}

\vspace*{0.4cm}

{\bf Parameters:}
\begin{itemize}
\item Dimension $N\in\NN$.
\item Number of frame elements $M\in\NN$.
\item Eigenvalues $\lambda_1\geq\cdots\geq\lambda_N>0$ such that
$\sum_{n=1}^N\lambda_n=M$.
\end{itemize}

\vspace*{0.2cm}

{\bf Algorithm:}
\begin{itemize}
\item[1)] Set $k = 1$.
\item[2)] For $j=1,\ldots,N$ do
\item[3)] \hspace*{0.5cm} Repeat
\item[4)] \hspace*{1cm} $S=\{L\in\{1,\ldots,N-j+1\}\colon L\leq\sum_{i=j}^{j+L-1}\lambda_i\}$.
\item[5)] \hspace*{1cm} If $S=\emptyset$ or if $M-k=N-j$ and 
$\sum_{i=1}^{j+\min S -1}\lambda_i\not\in\NN$ then
\item[6)] \hspace*{1.5cm} $L=M-k+1$.
\item[7)] \hspace*{1.5cm} $\omega=\exp(\frac{2\pi i}{L})$.
\item[8)]\hspace*{1.5cm}	For $r=k,\ldots,k+L-1$ do
\item[9)]\hspace*{2cm}	$
					f_r=\sqrt{\frac{\lambda_j}{L}}\cdot
					\sum_{i=j}^N\omega^{(i-j)(r-k)}e_i$.
\item[10)]\hspace*{1.5cm}	end.
\item[11)] \hspace*{1.5cm} $\lambda_j = \lambda_{j+1}=\cdots = \lambda_N = 0$.
\item[12)] \hspace*{1cm} else
\item[13)] \hspace*{1.5cm} If $\min S=1$ then
\item[14)] \hspace*{2cm} $f_k = e_j$.
\item[15)] \hspace*{2cm} $k = k+1$.
\item[16)] \hspace*{2cm} $\lambda_j = \lambda_j - 1$.
\item[17)] \hspace*{1.5cm} else
\item[18)] \hspace*{2cm} $L=\min S$.
\item[19)] \hspace*{2cm} $\omega=\exp(\frac{2\pi i}{L})$.
\item[20)]\hspace*{2cm}	For $r=k,\ldots,k+L-1$ do
\item[21)]\hspace*{2.5cm}	$
					f_r=\sqrt{\frac{\lambda_j}{L}}\cdot
					\sum_{i=j}^{j+L-2}\omega^{(i-j)(r-k)}e_i
				$
\item[]\hspace*{2.5cm}	
				$\phantom{f_r=}
					+\sqrt{1-\frac{1}{L}\sum_{i=j}^{j+L-2}\lambda_j}\cdot\omega^{(L-1)(r-k)}e_{j+L-1}.
				$
\item[22)]\hspace*{2cm}	end.
\item[23)]\hspace*{2cm}	$k=k+L$.
\item[24)] \hspace*{2cm} $\lambda_j  = \lambda_{j+1}=\cdots = \lambda_{j+L-2} = 0$.
\item[25)] \hspace*{2cm} $\lambda_{j+L-1} =\sum_{i=j}^{j+L-1}\lambda_i-L$.
\item[26)] \hspace*{1.5cm} end.
\item[27)] \hspace*{0.5cm} until $\lambda_j = 0$.
\item[28)] end.
\end{itemize}

\vspace*{0.2cm}

{\bf Output:}
\begin{itemize}
\item Unit norm frame $(f_m)_{m=1}^M\subseteq\CC^{N}$.
\end{itemize}
\vspace*{0.1cm}
\end{minipage}
}
\vspace*{0.2cm}
\caption{The DFTST algorithm for constructing a unit norm frame with prescribed spectrum in $(0,\infty)$.}
\label{fig:mostgeneral}
\end{table}

\begin{definition}
Given positive reals $\lambda_1,\ldots,\lambda_L$
we call a matrix a \emph{general $D_L$ block for $\lambda_1,\ldots,\lambda_L$}, if it is derived from the DFT matrix $F_L$ by multiplying the entries of the $j$-th row of $F_L$ by 
$\sqrt{\frac{\lambda_j}{L}}$ for $j=2,\ldots,L-1$ and the entries of the first and last row of 
$F_L$ by $\sqrt{\frac{c_1}{L}}$, respectively $\sqrt{\frac{c_L}{L}}$, where $0<c_1\leq\lambda_1$ and $0<c_L\leq\lambda_L$, and if it has normalized columns. We call $c_1$ and $c_L$ the 
\emph{first}, respectively \emph{last correction factor}.
\end{definition}

\begin{theorem}\label{finaltheorem}
Let $M\geq N$ be positive integers and $(\lambda_n)_{n=1}^N\subseteq(0,\infty)$ in decreasing order such that $\sum_{n=1}^{N}\lambda_n=M$. Then the algorithm DFTST constructs an $M$-element unit norm frame with eigenvalues
$(\lambda_n)_{n=1}^N$.
\end{theorem}
\begin{proof}
We construct the desired synthesis matrix successively one vector at a time by filling it with general $D_L$ blocks, similar as seen in example \ref{LRSTCexample}, only now each time we check for the minimal size of a general $D_L$ block that can be inserted. We do so by determining the set
$S$ in line $4)$ of the algorithm. Indeed, for $L\in S$ there does exist a general $D_L$ block
for $\lambda_j,\ldots,\lambda_{j+L-1}$, since the condition on $L$ is equivalent to
\[
1-\sum_{i=j}^{j+L-2}\frac{\lambda_i}{L}\leq\frac{\lambda_{j+L-1}}{L},
\]
and thus implies that the choice of $\lambda_j$ as the first correction factor results in a last correction factor, which is uniquely determined by the normalization condition of the columns, which is smaller or equal to $\lambda_{j+L-1}$. 

Since we desire sparsity as a key property of the synthesis matrix, we choose the smallest 
possible such $L$, i.e. the minimum of $S$, in case $S$ is not empty. If this minimum is $1$ we just insert a $1$, or in other words a $D_1$ block, i.e. a standard unit vector. This is done in lines $13)$ to $16)$ of the algorithm. Otherwise we insert a general $D_{\min S}$ block, see lines
$18)$ to $25)$.

The algorithm terminates after one performance of the lines $6)$ to $11)$. In lines $6)$ to $11)$,
the first $N-j+1$ rows of a general $D_{M-k+1}$ block are being inserted. This is done in case 
$S=\emptyset$, i.e. if we can not insert a general $D_L$ block for any $L\in\{1,\ldots,N-j+1\}$, or if
after inserting a general $D_{\min S}$ block, we
would not be able to continue a spectral tetris like construction due to the fact that we would now have more
rows than columns left to fill in the to be constructed synthesis matrix.
\end{proof}

\section{Fusion frames with prescribed eigenvalues and prescribed dimensions}\label{greater2}

Let $M\geq N$ be positive integers and $(\lambda_n)_{n=1}^N\subseteq(0,\infty)$ such that $\sum_{n=1}^N\lambda_n=M$. Given a sequence of dimensions, we ask the question of whether and how we can find a spectral tetris fusion frame for $\RR^N$ whose subspaces have those prescribed dimensions and whose fusion frame operator has the eigenvalues $(\lambda_n)_{n=1}^N$. 

To get started, consider the following example of integer eigenvalues:
\[
(\lambda_n)_{n=1}^6=(4,4,3,3,2,2). 
\]
Given this sequence of eigenvalues, the spectral tetris frame in $\RR^6$ consists only of standard unit vectors:
\[
S=\{e_1,e_1,e_1,e_1,e_2,e_2,e_2,e_2,e_3,e_3,e_3,e_4,e_4,e_4,e_5,e_5,e_6,e_6\}.
\]
The question we are asking above now takes the following form. We want to partition $S$ into sets of pairwise orthonormal vectors, i.e. each set of the partition should not contain more than one copy of any standard unit vector. What sizes can these sets have? We start by considering the partition $S=\bigcup_{n=1}^4P_n$, where
\begin{align*}
P_1&=\{e_1,e_2,e_3,e_4,e_5,e_6\},\\
P_2&=\{e_1,e_2,e_3,e_4,e_5,e_6\},\\
P_3&=\{e_1,e_2,e_3,e_4\},\\
P_4&=\{e_1,e_2\}.
\end{align*}
The sets of this partition have sizes $6,6,4$ and $2$. To get a different partition we cannot take any vector from $P_i$ and put it into $P_j$ if $i>j$, since this would destroy the orthonormality of the sets. We can on the other hand take certain vectors out of $P_i$ and put them into $P_j$ if $i<j$, without destroying the orthonormality of the sets. By doing so, we can for example easily find a partition into orthonormal sets of sizes $6,5,4$ and $3$. But, it is not possible to find a partition into orthonormal sets of sizes $6,6,5$ and $1$. The sequence $6,6,4,2$ majorizes the sequences of sizes of orthonormal sets which we can partition $S$ into. Let us recall the notion of majorization. 

\begin{definition}
Given $a=(a_n)_{n=1}^N\in\RR^N$, denote by $a^{\downarrow}\in\RR^N$ the vector obtained by rearranging the coordinates of $a$ in decreasing order. If $(a_n)_{n=1}^N,(b_n)_{n=1}^N\in\RR^N$, we say $(a_n)_{n=1}^N$ \emph{majorizes} $(b_n)_{n=1}^N$, denoted by $(a_n)\succeq (b_n)$, if $\sum_{n=1}^ma_n^{\downarrow}\geq\sum_{n=1}^mb_n^{\downarrow}$ for all $m=1,\ldots,N-1$ and $\sum_{n=1}^Na_n=\sum_{n=1}^Nb_n$.
\end{definition}

We will also use the notion of majorization between tuples of different length,  by agreeing to add zero entries to the shorter tuple, in order to have tuples of the same length.

We can use the idea of the above example to construct spectral tetris fusion frames in the general case of non integer eigenvalues. As above, we will determine a sequence of numbers depending on the given eigenvalues $(\lambda_n)_{n=1}^N$ and check whether or not this sequence majorizes the given sequence of dimensions. Again, the sequence we are going to determine will be the sequence of dimensions of a certain fusion frame for $\RR^N$ having the eigenvalues $(\lambda_n)_{n=1}^N$. We now introduce this fusion frame.

\begin{definition}
Let $M\geq N$ be positive integers and let $(\lambda_n)_{n=1}^N\subseteq(0,\infty)$ have the property that $\sum_{n=1}^N\lambda_n=M$. The fusion frame constructed by the algorithm RFF presented in table~\ref{fig:RFF} is called the \emph{reference fusion frame} for the eigenvalues $(\lambda_n)_{n=1}^N$.
\end{definition}

\begin{table}[h]
\centering
\framebox{
\begin{minipage}[h]{6.0in}
\vspace*{0.3cm}
{\sc \underline{RFF: Reference Fusion Frame}}

\vspace*{0.4cm}

{\bf Parameters:}
\begin{itemize}
\item Eigenvalues $(\lambda_n)_{n=1}^N\subseteq(0,\infty)$ such that
$\sum_{n=1}^N\lambda_n=M\in\NN$.
\end{itemize}

\vspace*{0.2cm}

{\bf Algorithm:}
\begin{itemize}
\item[1)] 
Use appropriate spectral tetris algorithm for $(\lambda_n)_{n=1}^N$ to get frame\\ $F=(f_m)_{m=1}^M$.
\item[2)] $t=$ maximal support size of the rows of $F$.
\item[3)]  $S_i=\emptyset$ for $i=1,\ldots,t$.
\item[4)] $k = 0$.
\item[5)] Repeat
\item[6)] \hspace*{0.5cm} $k=k+1$.
\item[7)] \hspace*{0.5cm} $j=\min\{1\leq r\leq t\colon
                                            \supp f_k\cap \supp f_s=\emptyset\text{ for all } f_s\in S_r\}$.
\item[8)] \hspace*{0.5cm} $S_j=S_j\cup\{f_k\}$.
\item[9)]  until $k=M$.
\end{itemize}

\vspace*{0.2cm}

{\bf Output:}
\begin{itemize}
\item Fusion frame $(V_i)_{i=1}^t$, where $V_i=\spann S_i$ for $i=1,\ldots,t$.
\end{itemize}
\vspace*{0.01cm}
\end{minipage}
}
\vspace*{0.2cm}
\caption{The RFF algorithm for constructing the reference fusion frame.}
\label{fig:RFF}
\end{table}

An example of RFF is included in the remarks following the proof of Proposition~\ref{nontight}. 

We will now use the reference fusion frame for $(\lambda_n)_{n=1}^N$ to tackle the question, whether or not a fusion frame for $\CC^N$, with a certain fusion frame operator and certain dimensions of the subspaces, is constructible via spectral tetris. In case it is constructible, the proof describes an algorithm to construct it.
In the proof we will use the following notation.
\begin{definition}
Let $S$ be a set of vectors in $\CC^N$, and $s\in S$. We say that a subset $C\subseteq S$ is a \emph{chain in $S$ starting at $s$}, if $s\in S$ and the support of any element in $S$ intersects the support of some other element of $S$. We say that $C$ is a \emph{maximal chain in $S$ starting at $s$} if $C$ is not a proper subset of any other chain  in $S$ starting at $s$.
\end{definition}

\begin{theorem}
\label{nontight}
Let $M\geq N$ be positive integers, $(\lambda_n)_{n=1}^N\subseteq(0,\infty)$ and let $(d_i)_{i=1}^D\subseteq\NN$ such that $\sum_{n=1}^N\lambda_n=\sum_{n=1}^Dd_n=M$. Let $(V_n)_{n=1}^t$ be the reference fusion frame for $(\lambda_n)_{n=1}^N$. If $(\dim V_n)\succeq(d_n)$, then there exists a spectral tetris fusion frame $(W_n)_{n=1}^D$ for $\CC^N$ with $\text{dim }W_n=d_n$ for $n=1,\ldots,D$ and eigenvalues $(\lambda_n)_{n=1}^N$.
\end{theorem}

\begin{proof}
We show how to iteratively construct the desired fusion frame $(W_n)_{n=1}^D$ in case the majorization condition holds.
Let $t$ and $S_1,\ldots,S_t$ be given by RFF for $(\lambda_n)_{n=1}^N$.
Let $W^0_i=S_i$ for $i=1,\ldots,t$. We add empty sets if necessary to obtain a collection $(W^0_i)_{i=1}^D$ of $D$ sets. If $\sum_{i=1}^{D}||W^0_i|-d_i|=0$ then the sets $(W^0_i)_{i=1}^{D}$ span the desired fusion frame. Otherwise, starting from $(W^0_i)_{i=1}^{D}$, we will construct the spanning sets of the desired fusion frame. 

Let
\[
m=\max\left\{j\leq D\colon d_j\neq |W^0_j|\right\}.
\]
Note that $\sum_{i=1}^{m}|W^0_i|=\sum_{i=1}^{m}d_i$ by the choice of $m$, and $\sum_{i=1}^{m-1}|W^0_i|>\sum_{i=1}^{m-1}d_i$ by the majorization assumption. Therefore, $d_m>|W^0_m|$ and there exists
\[
k=\max\left\{j<m\colon |W^0_j|>d_j\right\}.
\]
Notice that $|W^0_m|<d_m\leq d_k<|W^0_k|$ implies $|W^0_m|+2\leq|W^0_{k}|$. 

We now have to consider two cases. First, if there is some element $w\in W^0_{k}$, which has disjoint support from every element in $W^0_m$, define $(W_i^1)_{i=1}^{D}$ by
\begin{align}\label{c1}
W_i^1=
\begin{cases}
W^0_k\setminus\{w\}& \text{if } i=k,\\
W^0_{m}\cup\{w\}& \text{if } i=m,\\
W^0_i &\text{else.}
\end{cases}
\end{align}
Second, suppose there is no such element in $W^0_k$. 
Partition $W_k^0\cup W_m^0$ into maximal chains $C_1,\ldots,C_r$, say.
Note that for $i=1,\ldots,r$, the sizes of the sets $C_i\cap W_k^0$ and $C_i\cap W_m^0$ differ by at most one, since,
given $v_k\in W_k^0$ and $v_m\in W_m^0$, we know that $v_k$ and $v_m$ either have disjoint support, or their support sets have intersection of size one.
Since $|W^0_m|+2\leq|W^0_{k}|$, there is a maximal chain $C_j$ that contains one element more from $W_k^0$ than from $W_m^0$. Define $(W_i^1)_{i=1}^{D}$ by
\begin{align}\label{c2}
W_i^1=
\begin{cases}
(W^0_k\cup C_j)\setminus (C_j\cap W_k^0)& \text{if } i=k,\\
(W^0_m\cup C_j)\setminus (C_j\cap W_m^0)& \text{if } i=m,\\
W^0_i &\text{else.}
\end{cases}
\end{align}

In both of the above cases (\ref{c1}) and (\ref{c2}), we have defined
$(W_i^1)_{i=1}^{D}$ such that
\[
\sum_{i=1}^{D}||W^1_i|-d_i|<\sum_{i=1}^{D}||W^0_i|-d_i|.
\]
Note that $(W^1_i)_{i=1}^{D}$ satisfies the majorization condition in the sense that $(|W_n^1|)\succeq(d_n)$. Thus if the sets of $(W^1_i)_{i=1}^{D}$ do not span the desired fusion frame, we can repeat the above procedure with $(W^1_i)_{i=1}^{D}$ instead of $(W^0_i)_{i=1}^{D}$ and get $(W^2_i)_{i=1}^{D}$ such that $\sum_{i=1}^{D}||W^2_i|-d_i|<\sum_{i=1}^D||W^1_i|-d_i|$. Continuing in this fashion we will, say after repeating the process $l$ times, arrive at $(W^l_i)_{i=1}^{D}$ such that $\sum_{i=1}^{D}||W^l_i|-d_i|=0$, i.e. the sets of $(W^l_i)_{i=1}^{D}$ span the desired fusion frame $(W_n)_{n=1}^D$.
\end{proof}

Intuition suggests that for a given choice of eigenvalues, the dimensions derived from RFF for these eigenvalues in blockwise order, majorize the dimensions derived from RFF for the same eigenvalues in non-blockwise order. We do not investigate in this direction, since even two different blockwise orderings of the eigenvalues will in general lead to different sequences of dimensions of the reference fusion frame, as the following example shows. Given the eigenvalues  $\frac{5}{2},\frac{10}{3},\frac{13}{6}$, STC produces the synthesis matrix
\begin{align*}
F=[f_1\cdots f_8]=
\begin{bmatrix}
1&1&\sqrt{\frac{1}{4}}&\sqrt{\frac{1}{4}}&0&0&0&0\\
0&0&\sqrt{\frac{3}{4}}&-\sqrt{\frac{3}{4}}&1&\sqrt{\frac{5}{12}}&\sqrt{\frac{5}{12}}&0\\
0&0&0&0&0&\sqrt{\frac{7}{12}}&-\sqrt{\frac{7}{12}}&1
\end{bmatrix}
\end{align*} 
and thus the reference fusion frame for $(\frac{5}{2},\frac{10}{3},\frac{13}{6})$ is
\begin{align*}
V=(\spann\{f_1,f_5,f_8\},\spann\{f_2,f_6\},\spann\{f_3\},\spann\{f_4\},\spann\{f_7\}).
\end{align*}
Running RFF for a different blockwise ordering of the same eigenvalues, say $\frac{10}{3},\frac{5}{2},\frac{13}{6}$, yields the synthesis matrix
\begin{align*}
G=[g_1\cdots g_8]=
\begin{bmatrix}
1&1&1&\sqrt{\frac{1}{6}}&\sqrt{\frac{1}{6}}&0&0&0\\
0&0&0&\sqrt{\frac{5}{6}}&-\sqrt{\frac{5}{6}}&\sqrt{\frac{5}{12}}&\sqrt{\frac{5}{12}}&0\\
0&0&0&0&0&\sqrt{\frac{7}{12}}&-\sqrt{\frac{7}{12}}&1
\end{bmatrix}
\end{align*} 
and thus the reference fusion frame for $(\frac{10}{3},\frac{5}{2},\frac{13}{6})$ is
\begin{align*}
W=(\spann\{g_1,g_6\},\spann\{g_2,g_7\},\spann\{g_3,g_8\},\spann\{g_4\},\spann\{g_5\}).
\end{align*}
Note that RFF is in both cases performed for a blockwise ordering of the given eigenvalues, yet the sequences of dimensions of the subspaces of the reference fusion frames $V$ and $W$ are different.  

In the case of constructing tight spectral tetris frames, questions of how to order the eigenvalues before performing
the algorithm do not arise. Nevertheless, for
the construction of tight spectral tetris frames of redundancy less than two, the majorization condition presented in theorem~\ref{nontight} 
is not necessary either. Consider for example the case of a $10$-element unit norm tight frame in $\CC^7$. The synthesis
matrix TDFTST constructs, is composed of two blocks $D_2$ followed by two blocks $D_3$ and the reference fusion frame derived from it has
dimensions $(2,2,2,2,1,1)$. We might on the other hand construct a synthesis matrix for a 
$10$-element unit norm tight frame in $\CC^7$ by building its synthesis matrix using a $D_2$ block, followed by a $D_3$ block, followed by another $D_2$ and $D_3$ block. The reference fusion frame derived from this synthesis matrix has 
dimensions $(2,2,2,2,2)$.

We now show, that for constructing tight fusion frames of redundancy greater or equal to two, the majorization condition is also sufficient.
In this situation, again, questions of how to order the eigenvalues do not arise. Moreover, the 
vectors in the spanning sets that RFF produces, are either standard unit vectors or linear combinations of two consecutive standard unit vectors.

\begin{definition}
We call any standard unit vector $e_i\in\RR^N$ a \emph{singleton}. We call a linear combination $ae_i+be_{i+1}$ of two consecutive singletons a \emph{doubleton} and denote it by $e'_i$.
\end{definition}

By the structure of the synthesis matrix that spectral tetris produces and the fact that RFF sweeps through the synthesis matrix from left to right in order to fill the spanning sets of the reference fusion frame, we see that RFF picks singletons first, i.e. if for some $i=1,\ldots,N-1$  we have $e_i\in S_{n_1}$ and $e'_i\in  S_{n_2}$, then $n_1<n_2$. 

\begin{theorem}
Let $M\geq 2N$ be positive integers and $(d_n)_{n=1}^D\subseteq\NN$ such that $\sum_{n=1}^Dd_n=M$. Let $(V_n)_{n=1}^t$ be the reference fusion frame for $(\lambda_n)_{n=1}^N=(\frac{M}{N},\ldots,\frac{M}{N})$. Then there exists a tight spectral tetris fusion frame $(W_n)_{n=1}^D$ for $\mathbb{R}^N$ with $\text{dim }W_n=d_n$ for $n=1,\ldots,D$ if and only if $(\dim V_n)\succeq(d_n)$.
\end{theorem}

\begin{proof}
By Proposition~\ref{nontight} it remains to show that the majorization condition is necessary. To show this, let \smash{$(S_n)_{n=1}^t$} be the spanning sets of the reference fusion frame, and suppose 
\smash{$(F_n)_{n=1}^D$} is some other partition of the frame vectors STC generates for $(\lambda_n)_{n=1}^N$ into sets of orthogonal vectors. It suffices to define a sequence of spanning sets of fusion frames \smash{$(S^j_n)_{n=1}^{D_j}$}, $j=1,\ldots,r$ for some $r\in\NN$, such that \smash{$(S^0_n)_{n=1}^{D_0}=(F_n)_{n=1}^D$}, $(S^r_n)_{n=1}^{D_r}=(S_n)_{n=1}^t$ and $(|S^{j+1}_n|)\succeq(|S^j_n|)$ for all $j=0,\ldots,r-1$. We show how to construct such a sequence successively, i.e. we describe how to construct  \smash{$(S^{j+1}_n)_{n=1}^{D_{j+1}}$} from  \smash{$(S^j_n)_{n=1}^{D_j}$}. 

So let \smash{$(S^0_n)_{n=1}^{D_0}=(F_n)_{n=1}^D$},  fix $j$ and suppose  \smash{$(S^0_n)_{n=1}^{D_0},\ldots,(S^j_n)_{n=1}^{D_j}$} have already been constructed. If \smash{$(S^j_n)_{n=1}^{D_j}=(S_n)_{n=1}^t$}, then $j=r$ and we are done. Thus, we may assume that there exists
\[
n_0=\min\{n\leq D_j\colon S_n^j\neq S_n\}.
\]
Let $m_0$ be the minimal integer in $\{1,\ldots,N\}$, for which $S_{n_0}$ and $S_{n_0}^j$ differ, i.e. for which one of the following holds:
\begin{itemize}
\item[(i)] No vector in $S_{n_0}$  is supported at $m_0$, but some vector in $S_{n_0}^j$ is supported at $m_0$.
\item[(ii)] $e_{m_0}\in S_{n_0}$, but $e_{m_0}\not\in S_{n_0}^j$.
\item[(iii)] $e'_{m_0}\in S_{n_0}$, but $e'_{m_0}\not\in S_{n_0}^j$.
\end{itemize}
In each of these three cases we will describe how to construct  \smash{$(S^{j+1}_n)_{n=1}^{D_{j+1}}$} from 
\smash{$(S^j_n)_{n=1}^{D_j}$}, 
such that $(|S^{j+1}_n|)\succeq(|S^j_n|)$ and $S^{j+1}_{n_0}$ and  $S_{n_0}$ either both contain no vector supported at $m_0$ or the same vector supported at $m_0$. Note that by iterating this  procedure, we will thus, after a finite number of steps, arrive at some \smash{$(S^r_n)_{n=1}^{D_r}$} which is identical to \smash{$(S_n)_{n=1}^t$}. We now go through the three cases.

Case (i): This case can actually not occur, i.e. it is not possible that $S_{n_0}$ contains no vector supported at $m_0$ but that $S^j_{n_0}$ does. Indeed,  $S^j_{n_0}$ cannot contain a doubleton $e'_{m_0-1}$ by the minimality of $m_0$, and it cannot contain a doubleton $e'_{m_0}$ or the singleton $e_{m_0}$, since then RFF would have also picked just the same doubleton, or singleton respectively, for $S_{n_0}$.

Case (ii): Note that by the minimality of $m_0$, the reason for this case cannot be that $S^j_{n_0}$ contains a doubleton $e'_{m_0-1}$. Thus, there are two subcases which could have created case (ii).
\begin{itemize}
\item[(ii-a)] $S^j_{n_0}$ does not contain any vector supported at  $m_0$.
\item[(ii-b)] $S^j_{n_0}$ contains a doubleton $e'_{m_0}$.
\end{itemize}

Case (ii-a): There must be some $k> n_0$, such that $e_{m_0}\in S^j_k$ and we define
\smash{$(S_n^{j+1})_{n=1}^{D_{j+1}}$} by letting 
\begin{align*}
S^{j+1}_{i}=
\begin{cases}
S^{j}_{n_0}\cup\{e_{m_0}\}& \text{if } i=n_0,\\
S^{j}_{k}\setminus\{e_{m_0}\}& \text{if } i=k,\\
S^{j}_{n}& \text{else.} 
\end{cases}
\end{align*}

Case (ii-b): Again, identify the $k> n_0$, for which  $e_{m_0}\in S^j_k$. Now,  one of three things can happen. 

First, if $S^j_k$ contains the singleton $e_{m_0+1}$, define \smash{$(S_n^{j+1})_{n=1}^{D_{j+1}}$} by
letting 
\begin{align}\label{switching}
S^{j+1}_{i}=
\begin{cases}
(S^{j}_{n_0}\setminus\{e'_{m_0}\})\cup\{e_{m_0},e_{m_0+1}\}& \text{if } i=n_0,\\
(S^{j}_{k}\setminus\{e_{m_0},e_{m_0+1}\})\cup\{e'_{m_0}\}& \text{if } i=k,\\
S^{j}_{n}& \text{else.} 
\end{cases}
\end{align}
(To shorten our notation, in what follows we will use the following terminology for how we constructed \smash{$(S_n^{j+1})_{n=1}^{D_{j+1}}$} in
(\ref{switching}). We say we constructed  \smash{$(S_n^{j+1})_{n=1}^{D_{j+1}}$}from  \smash{$(S_n^{j})_{n=1}^{D_{j}}$} by switching $e'_{m_0}$ with $\{e_{m_0},e_{m_0+1}\}$ between $S^{j}_{n_0}$ and $S^{j}_k$.) 

Second, if $S^j_k$ contains no vector supported at $m_0$, define \smash{$(S_n^{j+1})_{n=1}^{D_{j+1}}$} by switching $e'_{m_0}$ with $e_{m_0}$ between $S^{j}_{n_0}$ and $S^{j}_k$. 

Third, $S^j_k$ contains a doubleton $e'_{m_0+1}$. This third subcase again has three subcases. The first being that $S^j_{n_0}$ contains $e_{m_0+2}$. In this subcase, we construct \smash{$(S_n^{j+1})_{n=1}^{D_{j+1}}$} by switching  $\{e'_{m_0}, e_{m_0+2}\}$ with $\{e_{m_0}, e'_{m_0+1}\}$ between $S^{j}_{n_0}$ and $S^{j}_k$. The second being that $S^j_{n_0}$ contains no vector supported at $m_0+2$. In this subcase, we construct \smash{$(S_n^{j+1})_{n=1}^{D_{j+1}}$} by switching  $e'_{m_0}$ with $\{e_{m_0}, e'_{m_0+1}\}$ between $S^{j}_{n_0}$ and $S^{j}_k$. The third being that $S^j_{n_0}$ contains a doubleton $e'_{m_0+2}$. This third subcase will again have three sub-subcases, namely that $S^j_k$ contains $e_{m_0+3}$, no vector supported at  $m_0+3$ or a doubleton $e'_{m_0+3}$. In the first two of those sub-subcases we can again define \smash{$(S_n^{j+1})_{n=1}^{D_{j+1}}$} by switching vectors in the above fashion and the third sub-subcase will create three new further cases. We can continue our argument successively,  considering three new subcases. Eventually, we must arrive at a point where we can define \smash{$(S_n^{j+1})_{n=1}^{D_{j+1}}$} by switching vectors, since we only deal with a finite number of vectors.

Case (iii): In this case it is not possible that $S^j_{n_0}$ contains $e_{m_0}$, since RFF picks singletons first and thus we would have $e_{m_0}\in S_{n_0}$ instead of $e'_{m_0}\in S_{n_0}$. By the minimality of $m_0$, it is also not possible for $S^j_{n_0}$ to contain some $e'_{m_0-1}$. Hence, $S^j_{n_0}$ contains no vector supported at $m_0$. There are three subcases which could occur under this circumstances.
\begin{itemize}
\item[(iii-a)] $S^j_{n_0}$ does not contain any vector supported at  $m_0+1$.
\item[(iii-b)] $e_{m_0+1}\in S^j_{n_0}$.
\item[(iii-c)] $e'_{m_0+1}\in S^j_{n_0}$.
\end{itemize}

Case (iii-a): Identify $k>n_0$ such that $e'_{m_0}\in S^j_k$ and define  
\smash{$(S_n^{j+1})_{n=1}^{D_{j+1}}$} by letting 
\begin{align*}
S^{j+1}_{i}=
\begin{cases}
S^{j}_{n_0}\cup\{e'_{m_0}\}& \text{if } i=n_0,\\
S^{j}_{k}\setminus\{e'_{m_0}\}& \text{if } i=k,\\
S^{j}_{n}& \text{else.} 
\end{cases}
\end{align*}

Case (iii-b): Identify $k> n_0$ such that $e'_{m_0}\in S^j_k$ and define 
\smash{$(S_n^{j+1})_{n=1}^{D_{j+1}}$} by switching $e_{m_0+1}$ with $e'_{m_0}$ between $S^{j}_{n_0}$ and $S^{j}_k$. 

Case (iii-c): This subcase has three new subcases in the same fashion as above, two of which lead to the definition on  \smash{$(S_n^{j+1})_{n=1}^{D_{j+1}}$} by switching vectors and the third of which again leads to three new subcases. So again we can continue our argument successively, considering three new subcases potentially over and over again, until eventually we arrive at a point where we can define \smash{$(S_n^{j+1})_{n=1}^{D_{j+1}}$} by switching vectors, since we only deal with a finite number of vectors.
\end{proof}

In the trivial case of integer eigenvalues we can run STC and make the following observation.

\begin{corollary}
If $(\lambda_n)_{n=1}^N\subseteq\NN$ and $(d_i)_{i=1}^D\subseteq\NN$ such that $\sum_{n=1}^N\lambda_n=\sum_{n=1}^Dd_n=M\in\NN$, where $M\geq N$, then a spectral tetris fusion frame with eigenvalues $(\lambda_n)_{n=1}^N$ and dimensions $(d_n)_{n=1}^D$ exists if and only if  $(a_n)\succeq(d_n)$, where $a_n=\max\{r\colon \lambda_r\geq n\}$ for $n=1,\ldots,\max_{i=1,\ldots,N}\lambda_i$.
\end{corollary}

\begin{proof}
Note that in this case RFF produces the output $t=\max_{i=1,\ldots,N}\lambda_i$ and that $\dim V_i=a_i$ for $i=1,\ldots,t$.
\end{proof}


\begin{thebibliography}{99}

\bibitem{Bod07}
B.~G.~Bodmann,
{\em Optimal Linear Transmission by Loss-Insensitive Packet Encoding},
Appl.~Comput.~Harmon.~Anal.~{\bf 22(3)} (2007), 274--285.


\bibitem{CFMPS11} J.~Cahill, M.~Fickus, D.~G.~Mixon, M.~Poteet, and N.~K.~Strawn,
 {\em Constructing finite frames of a given spectrum and set of lengths}, in submission.

\bibitem{CCHKP10} R.~Calderbank, P.~G.~Casazza, A.~Heinecke, G.~Kutyniok, and A.~Pezeshki,
 {\em Sparse fusion frames: existence and construction}, Adv.~Comput.~Math.~{\bf 35(1)} (2011), 1--31.


\bibitem{CHKWZ11}
P.~G.~Casazza, A.~Heinecke, K.~Kornelson, Y.~Wang, Z.~Zhou,
{\em Necessary and sufficient conditions to perform spectral tetris}, preprint.

\bibitem{CHKK10}
P.~G.~Casazza, A.~Heinecke, F.~Krahmer, and G.~Kutyniok,
 {\em Optimally sparse frames}, IEEE Trans. Inform. Theory 57 (2011), 7279--7287. 

\bibitem{CFMWZ09}
P.~G.~Casazza, M.~Fickus, D.~Mixon, Y.~Wang, and Z.~Zhou,
 {\em Constructing tight fusion frames}, Appl.~Comput.~Harmon.~Anal.~{\bf 30} (2011) 175--187.

\bibitem{CK04}
P.~G.~Casazza and G.~Kutyniok, {\em Frames of subspaces}, Wavelets, frames and operator theory,
Contemp.~Math., vol.~345, Amer.~Math.~Soc., Providence, RI, 2004, 87--113.

\bibitem{CK07b}
P.~G.~Casazza and G.~Kutyniok,
{\em Robustness of Fusion Frames under Erasures of Subspaces and of Local Frame Vectors},
Radon transforms, geometry, and wavelets (New Orleans, LA, 2006), 149--160, Contemp.~Math.~{\bf 464}, Amer.
Math.~Soc., Providence, RI, 2008.

\bibitem{CKL08}
P.~G.~Casazza, G.~Kutyniok, and S.~Li, {\em Fusion frames and distributed processing}, Appl.
Comput.~Harmon.~Anal.~{\bf 25(1)} (2008), 114--132.

\bibitem{KPCL08}
G.~Kutyniok, A.~Pezeshki, A.~R.~Calderbank, and T.~Liu,
{\em Robust Dimension Reduction, Fusion Frames, and Grassmannian Packings},
Appl.~Comput.~Harmon.~Anal.~{\bf 26(1)} (2009), 64--76.

\bibitem{MRS10}
P.~G.~Massey, M.~A.~Ruiz, and D.~Stojanoff,
{\em The Structure of Minimizers of the Frame Potential on Fusion Frames},
J.~Fourier Anal.~Appl.~{\bf 16} (2010), 514--543.

\bibitem{PKC08}
A.~Pezeshki, G.~Kutyniok, and R.~Calderbank,
{\em Fusion frames and Robust Dimension Reduction,}
42nd Annual Conference on Information Sciences and Systems (CISS),
Princeton University, Princeton, NJ, Mar.~19--21, 2008.

\end{thebibliography}
\end{document}